\documentclass[11pt]{article}
\usepackage{amsmath,amssymb,mathdots}
\usepackage{latexsym}
\usepackage{datetime}
\usepackage[all]{xy}

\def\e#1\e{\begin{equation}#1\end{equation}}
\def\ea#1\ea{\begin{align}#1\end{align}}

\newtheorem{thm}{Theorem}[section]
\newtheorem{lem}[thm]{Lemma}
\newtheorem{prop}[thm]{Proposition}

\newtheorem{cor}[thm]{Corollary}
\newenvironment{dfn}{\medskip\refstepcounter{thm}
\noindent{\bf Definition \thesection.\arabic{thm}\ }}{\medskip}

\newenvironment{proof}[1][,]{\smallskip\ifcat,#1
\noindent{\it Proof.\ }\else\noindent{\it Proof of #1.\ }\fi}
{\relax\unskip\nobreak ~\hfill$\square$\bigskip}
\def\narrow{\par\addtolength{\leftskip}{20pt}\addtolength{\rightskip}{5pt}}
\def\oldmargins{\par\setlength{\leftskip}{0pt}\setlength{\rightskip}{0pt}}
\newcounter{quest}[section]
\newcounter{qa}[quest]
\newcounter{qi}[quest]

\def\inext{\par\ifnum\value{qi}=0 \narrow\fi \medskip\addtocounter{qi}{1}
\noindent\hbox to 0pt{\hss\bf(\roman{qi})\hskip .3em}}
\def\anext{\par\ifnum\value{qa}=0 \narrow\fi \medskip\addtocounter{qa}{1}
\noindent\hbox to 0pt{\hss\bf(\alph{qa})\hskip .3em}}

\numberwithin{equation}{section}

\newcommand{\be}{\begin{equation}}
\newcommand{\ee}{\end{equation}}
\newcommand{\bes}{\begin{equation*}}
\newcommand{\ees}{\end{equation*}}
\newcommand{\bea}{\begin{eqnarray}}
\newcommand{\eea}{\end{eqnarray}}

\def\w{\wedge}
\def\bw{\bigwedge}
\def\pa{\partial}
\def\bpa{{\bar \partial}}

\def\dpdl{d+\dl}

\def\ds{d^*}
\def\dl{d^\Lambda}
\def\dls{d^\Lambda{}^*}

\def\ss{*_s}

\def\hB{{\widehat B}}

\def\CA{\Omega}
\def\CB{\mathcal{P}}
\def\CH{\mathcal{H}}
\def\CJ{\mathcal{J}}
\def\CL{\mathcal{L}}

\def\A{\mathcal{A}}

\def\La{\Lambda}

\def\a{\alpha}
\def\b{\beta}

\def\om{\omega}
\def\s{\sigma}

\def\im{{\rm im~}}
 
\def\ea{e_1}
\def\eb{e_2}
\def\ec{e_3}
\def\ed{e_4}
\def\eee{e_5}
\def\ef{e_6}

\def\ars{a_{r,l}}
\def\dpp{{\partial_+}}
\def\dpm{{\partial_-}}
\def\dppm{{\partial_{\pm}}}
\def\dpps{{\partial_{+}^*}}
\def\dpms{{\partial_{-}^*}}

\def\tA{{\tilde A}}
\def\tB{{\tilde B}}

\begin{document}



\title{\bf{Cohomology and Hodge Theory on\\
Symplectic Manifolds: II}}

\author{Li-Sheng Tseng and Shing-Tung Yau \\
\\
}

\date{November 4, 2010}

\maketitle

\begin{abstract} \

We show that the exterior derivative operator on a symplectic manifold has a natural decomposition into two linear differential operators, analogous to the Dolbeault operators in complex geometry.  These operators map primitive forms into primitive forms and therefore lead directly to the construction of primitive cohomologies on symplectic manifolds.  Using these operators, we introduce new primitive cohomologies that are analogous to the Dolbeault cohomology in the complex theory.  Interestingly, the finiteness of these primitive cohomologies follows directly from an elliptic complex.  We calculate the known primitive cohomologies on a nilmanifold and show that their dimensions can vary with the class of the symplectic form.

\end{abstract}

\tableofcontents 

\setcounter{equation}{0}
\setcounter{footnote}{0}

\section{Introduction}

This paper continues the study of differential cohomologies on smooth compact symplectic manifolds that we began in Paper I \cite{TY1}.  There, we introduced a number of new finite-dimensional cohomologies defined on the space of differential forms.  These new cohomologies, dependent on the symplectic form, were shown in general to be distinct from the de Rham cohomolgy, and thus they provide new symplectic invariants.  Of particular interest for us here is the property noted in \cite{TY1} that the new symplectic cohomologies can be equivalently described by cohomologies defined only on the subset of differential forms called primitive forms.   We called this type of cohomologies ``primitive cohomologies" and they are the main focus of this paper.

The fundamental nature of primitive cohomologies in symplectic geometry can be understood simply.  Let us explain this via an analogy with complex geometry.  

On a complex space, it is standard to decompose differential forms into its $(p,q)$ components.  Let $\A^{p,q}$ be the space of smooth differential $(p,q)$-forms.  Then acting on it by the exterior derivative $d$, we have 
\be\label{ade}
d:~ \A^{p,q} ~ \to ~~\A^{p+1,q} \oplus ~\A^{p,q+1}~, 
\ee
which of course encodes the complex decomposition, $d=\pa + \bpa\,$, where for instance the Dolbeault operator $\bpa: \A^{p,q} \to \A^{p,q+1}$ projects $d (\A^{p,q})$ onto its $(p,q+1)$ component.  Thus, on a complex space, it is natural to decompose both the differential forms and exterior derivative in a complex structure dependent way.   This raises the question in the symplectic context whether a symplectic structure dependent decomposition of differential forms and the exterior derivative are also possible.

For differential forms, the decomposition in the presence of a symplectic form $\om$ is well known \cite{Weil, LM}.  This is commonly called the Lefschetz decomposition.   The elemental components, which we shall label by two indices, $(r,s)$, take the form $\frac{1}{r!}\,\om^r \w B_s\,$, where $B_s\in \CB^s$ is a primitive $s$-form.  Recall that by definition, a primitive form satisfies $\Lambda \,B_s := \dfrac{1}{2}(\om^{-1})^{ij}\,i_{\pa_{x^i}} i_{\pa_{x^j}} B_s =0\,$.  
We shall denote the space of such $(r,s)$-forms by $\CL^{r,s}$ with $0\leq r,s \leq n $ for a symplectic space of dimension $d=2n\,$. In Fig. \ref{sympr}, we have arranged the different $\CL^{r,s}${}'s into a pyramid diagram, representing the symplectic analog of the complex $(p,q)$ diamond.

\begin{figure}
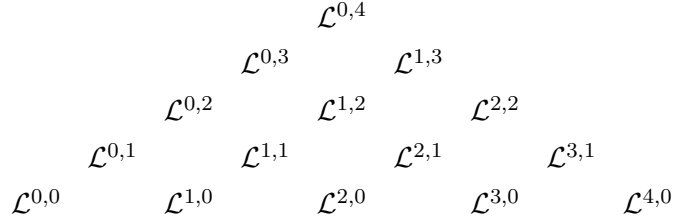

$$\begin{matrix}
 & & & & \CL^{0,4}& & & & \\
 & & &\CL^{0,3}& &\CL^{1,3} & & & \\
 & &\CL^{0,2}& &\CL^{1,2}& &\CL^{2,2} & & \\
 & \CL^{0,1} & & \CL^{1,1} & &\CL^{2,1}& &\CL^{3,1} & \\
\CL^{0,0}& &\CL^{1,0}& &\CL^{2,0}& &\CL^{3,0}& &\CL^{4,0}\\
\end{matrix}$$
\caption{The $(r,s)$ pyramid decomposition of differential forms in $d=8$.  The differential forms of $\CL^{r,s}$ has degree $(2r + s)$.  In the diagram, the degree increases in increment of one, from zero to $2n$ (left to right).}
\label{sympr}
\end{figure}

What may be a bit surprising is that a symplectic decomposition of the exterior derivative is also possible.   (As far as we are aware, this has not been previously discussed in the literature.)  Consider simply the action of $d$ on $\CL^{r,s}$.  Since $d\,\om =0\,$, we have $d\left(\frac{1}{r!}\,\om^r \w B_s\right)= \frac{1}{r!}\,\om^r \w (d B_s)\,$.  By this simple relation, we see clearly that the action of the exterior derivative on $\CL^{r,s}$ is entirely determined by its action on the primitive part, i.e. $B_s\,$.  And regarding the derivative of a primitive form, there is a useful formula (see, e.g., \cite{Huy, TY1} or Section 2.2 below \eqref{dprim}): 
\be\label{dbform}
dB_s = B^0_{s+1} + \om \w B^1_{s-1}~,
\ee
where $B^0, B^1 \in \CB^*\,$, and if $s=n$, then $B_{n+1}^0 = 0$.  Taking the exterior product of \eqref{dbform} with $\frac{1}{r!}\, \om^r$, we find 
\be \label{clde}
d: \CL^{r,s} \to \CL^{r,s+1} \oplus \CL^{r+1,s-1}~,
\ee 
which gives us the symplectic analog of \eqref{ade}.  

The symplectic decomposition of the exterior derivative becomes now just a projection onto the two different spaces on the right hand side of \eqref{clde}.  But as already mentioned, the derivative action on $\CL^{r,s}$ is completely encoded in its action on the primitive forms; thus, we really only need to consider the primitive components, $\CL^{0,s}=\CB^s$.  (A complete discussion taking into account of all $\CL^{r,s}$ will be given in Section 2.)  With \eqref{dbform}, we are led to write the decomposition of $d$ as follows: 
\be\label{ddecomp1}
d = \dpp \,+\; \om \w \dpm~,
\ee
where $\dppm: \CB^s \to \CB^{s\pm 1}$.  Hence, we have seen the importance of the primitive subspace of differential forms and have defined a pair of new first-order symplectic differential operators $(\dpp,\dpm)$ that preserve the primitive property of forms.   

Just like $(\pa,\bpa)$ on a complex space,  $(\dpp,\dpm)$ has a number of desirable properties that follow directly from their definition in \eqref{ddecomp1}.  In fact, it follows from $d^2=0$ and the two decompositions - Lefschetz and the exterior derivative \eqref{ddecomp1} - that both $\dpp$ and $\dpm$ square to zero and effectively anticommute with each other (see Lemma \ref{pmprop}).  These facts suggest defining on a symplectic manifold $(M^{2n}, \om)$ the following two cohomologies:
\be\label{ppcoh}
PH^k_\dpp(M) = \frac{\ker \dpp \cap \CB^k(M)}{\im \dpp \cap \CB^k(M)}~,
\ee
\be\label{pmcoh}
PH^k_\dpm(M) = \frac{\ker \dpm \cap \CB^k(M)}{\im \dpm \cap \CB^k(M)}~,
\ee
for $k< n\,$.  The two cohomologies are not well defined for $k=n$, since, by the definition of primitivity, there are no degree $n+1$ primitive forms.  Note that these new cohomologies $PH^*_\dpp(M)$ and $PH^*_\dpm(M)$ are very different from the de Rham cohomology.  For instance, a $ \dpp$-closed form may not be $d$-closed, and moreover, $d\,\dpp = \om \w (\dpm\dpp)\,$, which is not identically zero.  Nevertheless, we will show that the two cohomologies above are indeed finite dimensional on a compact manifold and in general they give new symplectic invariants.  Interestingly, the finiteness follows directly by associating the two cohomologies with the single differential complex, 
\be\label{ecomp}   
\xymatrix@R=20pt@C=20pt{ 0~\ar[r] & ~ \CB^{0} ~ \ar[r]^{\dpp} &~ \CB^{1} ~ \ar[r]^{\dpp} &~ ~\ldots~ ~\ar[r]^{\dpp}&~ \CB^{n-1}~ \ar[r]^{\dpp} &~ \CB^{n}~\ar[r]^{\dpp\dpm}&~\CB^{n}~\ar[r]^{\dpm} & {~~~~} & \\
 & & & {~~~~~}\ar@{^{(}->}[r]^\dpm& ~ \CB^{n-1}~\ar[r]^{\dpm}& ~~\ldots~~\ar[r]^{\dpm}&~ \CB^{1} ~ \ar[r]^{\dpm} &~ \CB^{0} ~ \ar[r]^{\dpm} &~ 0  }
\ee
which we will prove is elliptic (in Section 2, Proposition \ref{ellipticsc}).  This elliptic complex can be thought of as the symplectic analog of the Dolbeault complex.  Now, if we introduce a Riemannian metric, the elliptic complex implies that the second-order Laplacians, $\Delta_{\dppm}\,$, associated with $PH^*_\dppm(M)\,$ are also elliptic, and hence the primitive cohomologies have Hodge theoretic properties.  Moreover, we will also show that these two cohomologies are actually isomorphic, i.e., $PH^k_\dpp(M) \cong PH^k_\dpm(M)\,$.

The primitive cohomologies $PH^k_{\dppm}(M)$ also have an interesting alternative description. Let us denote the space of primitive $\dpm$-closed $k$-form by $\CB'^k(M)\,$ (with an additional prime).  By demonstrating the validity of the local $\dpm$-Poincar\'e lemma, we shall show the isomorphism of $PH^k_\dpm(M)$ with the C{e}ch cohomology ${\breve H}^{n-k}(M,\CB'^n)\,$, for $0\leq k < n\,$.  (This and other properties of $PH^k_{\dppm}(M)$ will be worked out in Section 3.) It is interesting to note here that the middle degree primitive $\dpm$-closed forms plays a special role.  As pointed out in \cite{TY1}, the Poincar\'e dual currents of closed lagrangians are precisely $d$-closed (or equivalently, $\dpm$-closed) middle degree primitive currents.  

\begin{table}[t]
\begin{center}
{\it Primitive Cohomologies}
\vspace{2mm}
{\renewcommand{\arraystretch}{2.2} 
\renewcommand{\tabcolsep}{0.3cm}
\begin{tabular}{l l | l l }

& ~~~~~~~~~$0\leq k\,<\,n$ &  & ~~~~~~~~~~~~~~~~~$0\leq k\, \leq \, n$\\
\hline
$(1)$& $PH^k_\dpp(M) = \dfrac{\ker \dpp \cap \CB^k(M)}{\im \dpp \cap \CB^k(M)}$& $(3)$
& $PH^k_{\dpdl}(M) = \dfrac{\ker (\dpp + \dpm) \cap \CB^k(M)}{\dpp\dpm \CB^k(M)}$  \\ 
$(2)$& $PH^k_\dpm(M) = \dfrac{\ker \dpm \cap \CB^k(M)}{\im \dpm \cap \CB^k(M)}$ & $(4)$
& $PH^k_{d\dl}(M) = \dfrac{\ker \dpp\dpm \cap \CB^k(M)}{ \dpp \CB^{k-1} + \dpm \CB^{k+1}}$ 
\end{tabular}}
\end{center}
\
 \caption{The primitive cohomologies defined on a symplectic manifold $(M^{2n}, \,\om)$ introduced here (1-2) and in Paper I (3-4) \cite{TY1}, expressed in terms of $\dpp$ and $\dpm\,$. 
\label{tabcoh}}
 \
 
\end{table}

Concerning the middle degree, observe that in the elliptic complex \eqref{ecomp} above, the middle degree primitive forms are curiously connected by the second-order differential operator, $\dpp\dpm\,$.  With its presence, two middle-dimensional primitive cohomologies can be read off from the elliptic complex: 
$$PH^n_{d\dl}= \frac{\ker \dpp\dpm \cap \CB^n(M)}{\im \dpp \cap \CB^n(M)}~, \qquad PH^n_{d+\dl}= \frac{\ker \dpm \cap \CB^n(M)}{\im \dpp\dpm\cap \CB^n(M)} ~.$$
These two middle-dimensional cohomologies are actually special cases of the two primitive cohomologies - $PH^k_{d+\dl}(M)$, $PH^k_{d\dl}(M)$
- introduced in Paper I \cite{TY1}.  These two were obtained by Lefschetz decomposing their corresponding symplectic cohomologies - $H^*_{d+\dl}(M)$, $H^*_{d\dl}(M)$
- which are defined on the space of all differential forms.  The two primitive cohomologies from Paper I are well defined for all $k\leq n$, which includes the middle degree, and can be expressed in terms of $\dpp$ and $\dpm\,$, as presented in Table \ref{tabcoh}, where we have collected the various primitive cohomologies.

We should emphasize that the dimensions of all the primitive cohomologies are invariant under symplectomorphisms.  However, they can vary along with the de Rham class of the symplectic form.  In Section 4, we calculate the various primitive symplectic cohomologies for a six-dimensional symplectic nilmanifold.  As can be seen clearly in Table \ref{extable} in Section 4, primitive cohomologies on a symplectic manifold do contain more information than the de Rham cohomology.  In particular, we will show explicitly that the dimension of $PH^2_\dppm(M)$ does vary in this specific nilmanifold as the class of the symplectic form varies.  

This paper for the most part will be focused on introducing $PH^k_{\dppm}(M)$ and developing their properties.  A fuller discussion of applications and relations between the different primitive cohomologies will be given elsewhere \cite{TY3}.



\medskip

\noindent{\it Acknowledgements.~} 
We would like to thank C.-Y. Chi, T.-J. Li, X. Sun, C. Taubes, C.-J. Tsai, and especially V. Guillemin for helpful comments and discussions.  This work is supported in part by NSF grants 0714648, 0804454, and 0854971.


\section{Primitive Cohomologies}

We begin by discussing the primitive structures on symplectic spaces that arise due to the presence of a symplectic form.  
We then proceed to develop the primitive cohomologies and show their finiteness on compact symplectic manifolds.

\subsection{Primitive structures on symplectic manifolds}

Let ($M^{2n}, \om$) be a smooth symplectic manifold.   There is a natural $sl_2$ representation $(L,\La, H)$ that acts on the space of differential forms, $\CA(M)$.  On a differential form $A\in \CA^*(M)$, the operators act as follows:
\begin{align*}
L &: ~~ L (A) =\om \w A ~, \\
\La &:~~ \La (A_k) =\frac{1}{2}(\om^{-1})^{ij}\,i_{\pa_{x^i}} i_{\pa_{x^j}} A  ~,\\
H &:~~ H(A) =  \sum_k (n-k)\, \Pi^k  A~,
\end{align*}
where $\w$ and $i$, respectively, denote the wedge and interior product, $(\om^{-1})^{ij}$ is the inverse matrix of $\om_{ij}$, and $\Pi^k:\CA^*(M) \to \CA^k(M)$ projects onto forms of degree $k$.  These actions result in the $sl_2$ algebra 
\be\label{Lalg}
[\La, L] = H\,, \qquad [H, \La]  = 2 \La\,,\qquad [H,L] = -2 L ~ .
\ee

With this $sl_2$ action, the space of all differential forms $\CA(M)$ can be arranged in terms of irreducible modules of the $sl_2$ representation \cite{Weil}.  From this perspective, the primitive forms are precisely the highest-weight representatives of the $sl_2$ modules.  More concretely, a differential $s$-form is called a primitive form, i.e. $B_s \in \CB^s(M)$ with $s\leq n$, if it satisfies the condition $\La B_s =0$, or equivalently, $L^{n-s+1} B_s =0$.

Now we can of course also decompose any differential form $A_k \in \CA^k(M)$ into components of different $sl_2$ modules.  This is commonly called the Lefschetz decomposition (from the K\"ahler geometry literature).  Specifically, we can write
\be\label{Lefdec}
A_k = \sum_{r\geq{\rm max}(k-n,0)}\frac{1}{r!}\, L^r B_{k-2r}~.
\ee
We emphasize that the Lefschetz decomposition is unique, as the $B_{k-2r}$'s in \eqref{Lefdec} above are solely determined by $A_k$.  By a straightforward calculation, we find 
\begin{align}\label{Bdef}
B_{k-2r} 
&=  \left(\sum_{l=0} \,\ars \,\frac{1}{l\,!}\,L^l \La^{r+l}\right) \, A_k~,
\end{align}
where $\ars$ are rational coefficients given by the expression
\begin{align}\label{arsf}
\ars
 = (-1)^l \, (n-k+ 2r +1)^2 \prod_{i=0}^r \frac{1}{n-k+2r+1 -i } \,\prod_{j=0}^l \frac{1}{n-k+2r +1 +j}~.
\end{align}
Thus, for example, it follows from \eqref{Bdef} and \eqref{arsf} that the first primitive form term $B_k$ in the decomposition $A_k = B_k + L\, B_{k-2} + \ldots~ $ for $k\leq n$ has the expression
\be\label{aprojb}
B_k= \left\{1 - \frac{1}{n-k+2} L\La + \frac{1}{2!} \frac{1}{(n-k+2)(n-k+3)} L^2 \La^2 - 
\ldots \right\} A_k~.
\ee

To fully appreciate the decomposition, it is useful to see the Lefschetz decomposition applied to all differential forms of a given dimension $d=2n$.  We write out the decomposition for $d=8$ in Fig. \ref{Lefdec8} having arranged the terms in a suggestive manner.  
\begin{figure}
\begin{align*}
A_0 &= B_{0,0} \\
A_1 &= ~~~~~~~~~~~B_{0,1} \\
A_2 &= L B_{1,0} ~~~~~~~~~~~~+ B_{0,2}  \\
A_3 &= ~~~~~~~~~~~ L B_{1,1} ~~~~~~~~~~~~~+B_{0,3} \\
A_4 &= \frac{1}{2!} L^2 B_{2,0}~~~~~~~~+ L B_{1,2}~~~~~~~~~~+B_{0,4}  \\
A_5 &= ~~~~~~~~~~~\frac{1}{2!} L^2 B_{2,1}~~~~~~~~~    + L B_{1,3}\\
A_6 &= \frac{1}{3!} L^3 B_{3,0}~~~~~~~~+ \frac{1}{2!} L^2 B_{2,2}  \\
A_7 &=  ~~~~~~~~~~~\frac{1}{3!} L^3 B_{3,1} \\
A_8 &= \frac{1}{4!} L^4 B_{4,0} 
\end{align*}
\caption{Lefschetz decomposition of differential forms in dimension $d=8$.  Here, $B_{r,k-2r}$ denotes a primitive $(k-2r)$-form associated with the $\frac{1}{r!}L^r$ term.  }
\label{Lefdec8}
\end{figure}

Clearly, each term of the decomposition can be labeled by a pair $(r,s)$ corresponding to the space
\be
\CL^{r,s}(M)=\left\{A \in \CA^{2r+s}(M){\big \arrowvert} A= \frac{1}{r!}L^r B_s\, ~{\rm with}~~ \La\, B_s =0 \right\}~.
\ee
Notice that the indices $r$ and $s$ each takes value between $0$ and $n$.  And we can naturally arrange all $\CL^{r,s}$'s into a {\it pyramid} as in Fig. \ref{sympr} (in Section 1), having rotated the terms in Fig. \ref{Lefdec8} counterclockwise by $90\,^{\circ}$.  The symplectic pyramid is heuristically for our purpose the analog of the $(p,q)$ diamond of complex geometry.

To distinguish the different $\CL^{r,s}(M)$ spaces, we shall make use of the operator $H$ and also introduce the operator $R\,$, which picks out the $r$ index.

\begin{dfn}\label{rdef} 
On a symplectic manifold, $(M,\om)$, the $R$ operator acts on an element $L^{r,s}\in \CL^{r,s}(M)$ as 
\be\label{rdefeq} 
R\, L^{r,s} = r \, L^{r,s}~.
\ee
\end{dfn}

The $s$ index is discerned by the $(H+2R)$ operator
\be\label{sfeq}
(H+2R)\, L^{r,s} = (n-s) L^{r,s} ~,
\ee 
where again $L^{r,s}\in \CL^{r,s}(M)\,$.
Note that acting on $\CL^{r,s}(M)$, $L$ and $\La$ raises and lowers $R$ by one, respectively.  More precisely, we have the following useful relations relating $(L,\La, H, R)$.  
\begin{lem}
On a symplectic manifold $(M,\om)$, the following relations hold:
\begin{itemize}
\setlength{\parsep}{0pt}
\setlength{\itemsep}{0pt}
\item[{\rm(i)}] $\left[\La\,,\, L^r\right] = (H+r-1)\,r\, L^{r-1}\,$ for $r\geq 1\,$;  
\item[{\rm(ii)}] $L\La = (H+R+1) R\,$;
\item[{\rm(iii)}] $\La L = (H+R) (R+1)\,$.
\end{itemize}
\label{llhrelations}
\end{lem}
\begin{proof}(i) follows straightforwardly from repeated applications of the $sl_2$ algebra commutation relations in \eqref{Lalg}.  (ii) and (iii) can be checked by acting on $\CL^{r,s}(M)$ and using (i).  

\end{proof}

Let us now introduce the symplectic star operator $\ss: \CA^k(M) \to \CA^{2n-k}(M)$ introduced in \cite{EH, Brylinski}.  It is defined by the local inner product  
\begin{align}
A \w \ss A' &=  (\om^{-1})^k(A, A') \,d\,{\rm vol} \notag\\
& = \frac{1}{k!}(\om^{-1})^{i_1 j_1}(\om^{-1})^{i_2 j_2}\!\ldots (\om^{-1})^{i_k j_k}\, A_{i_1 i_2 \ldots i_k}\,A'_{j_1 j_2 \ldots j_k}\,\,  \frac{\om^n}{n!}~. \label{sstardef}
\end{align}   
We note that $\ss\ss =1$, which follows from Weil's relation \cite{Weil,Guillemin}
\be \label{ssrel}
\ss \frac{L^r}{r!} B_s = (-1)^{s(s+1)/2} \frac{L^{n-r-s}}{(n-r-s)!} B_s~.
\ee
Therefore, acting on $\CL^{r,s}(M)$, we have that 
$$\ss: \CL^{r,s}(M) \to \CL^{n-r-s,s}(M)\,.$$  
In particular, for forms of middle degree $k= 2r+s =n$, or equivalently, $r= \frac{1}{2}(n-s)$, the action of the $\ss$ operator leaves them invariant up to a $-1$ factor.   And consider all $\CL^{r,s}(M)$ elements together as in the pyramid diagram in Fig. \ref{sympr}, the action of $\ss$ is a reflection with respect to the central vertical axis.

Finally, let us briefly discuss the linear structure - the primitive exterior vector space.  Let $V$ be a real symplectic vector space of dimension $d=2n\,$.  We write $\bw^k V$ for the $k$-exterior product of $V\,$.  Let $e_1, e_2,\ldots, e_{2n}$ be a basis for $V$ and take the symplectic form to be  $\om = e_1 \w e_2 + \ldots + e_{2n-1}\w e_{2n}\,$.   Let $P\bw^kV$ denote the primitive elements of $\bw^k V$.  The symplectic pyramid as in Fig. \ref{sympr} allows us to relate the dimension of $P\bw^k V$  with the dimension of $\bw^k V$.   Specifically, for $k\leq n$, it easy to see from the pyramid diagram that
\be
\dim P\bw{\!}^k \,V =  \dim \bw{\!}^k\, V - \dim \bw{\!}^{k-2}\, V = \binom{2n}{k} - \binom{2n}{k-2}~.
\ee
Moreover, the sum of the dimensions of all primitive exterior vector space is given by
\be
\sum_{k=0}^n \dim P\bw{\!}^k\, V =  \dim \bw{\!}^{n-1}\, V + \dim \bw{\!}^{n}\, V= \binom{2n}{n-1} + \binom{2n}{n}~.
\ee
Let us also give a canonical recursive method to write down the set of basis elements of $P\bw^kV\,$.  The idea is to construct the basis elements of dimension $d=2n$ from those of dimension $d=2(n-1)$.  For instance, selecting out the $e_1$ and $e_2$ elements, we have the following decomposition.
\begin{lem} Let $V$ be a symplectic vector space with the non-degenerate form $\om = e_{12} + e_{34} + \ldots + e_{2n-1,2n}$ (where the notation $e_{12}= e_1\w e_2$).  Then any element of the primitive exterior vector space $\mu_k\in P\bw^k V$ can be expressed as
\be\label{pvectdd}
\mu_k = e_1 \w \b_1 + e_2 \w \b_2 + (e_{12} -  \frac{1}{H+1} \sum_{j=2}^n e_{2j-1, 2j}) \w \b_3 + \b_4~,
\ee
where $\b_1, \b_2 \in P\!\bw^{k-1}\!V$, $\b_3 \in P\!\bw^{k-2}\!V$, $\b_4\in P\!\bw^{k}\!V$, and further $\b_1,\b_2,\b_3, \b_4$ do not contain either $e_1$ or $e_2\,$.\label{lpvectd}
\end{lem}
\begin{proof}
Generally, we can write 
\be\label{muorig}
\mu = e_1 \w \a_1 + e_2 \w \a_2 + e_{12} \w \a_3 + \a_4~,
\ee
where $\a_1,\a_2,\a_3, \a_4$ are exterior products of $e_3,e_4,\ldots ,e_{2n-1},e_{2n}\,$.  Primitivity of $\mu$ implies
\begin{align*}
\La\mu = e_1 \w \La \a_1 + e_2 \w \La \a_2 + e_{12} \w \La \a_3 + \a_3 + \La \a_4 =0 ~,
\end{align*}
giving the conditions
\begin{align}
\La \a_1 = \La \a_2 &= \La \a_3=0~,\label{aoth}\\
\a_3 + \La \a_4 &=0 ~.\label{athf}
\end{align}
Hence, $\a_i$ for $i=1,2,3$ must be primitive and we will denote these $\a_i$ by $\b_i\in P\!\bw^*\!V\,$ to highlight their primitive property.  Now, $\a_4$ is not primitive.  But with \eqref{athf} and $\a_3=\b_3$ being primitive, we have $\La^2 \a_4 =0$.  Thus, we can write  $\a_4 = \b_4 + (\om - e_{12}) \w \b'_4\,$ with $\b_4,\b'_4\in P\!\bw^*\!V\,$, and moreover, using \eqref{athf} again, we have $\b'_4 = - (H-1)^{-1} \b_3\,$.  Altogether, \eqref{muorig} becomes
\begin{align*}
\mu_k &= e_1 \w \b_1 + e_2 \w \b_2 + e_{12} \w \b_3 + \b_4 - (e_{34}+\ldots+e_{2n-1,2n})\, \frac{1}{(H-1)}\, \b_3 \\
&=e_1 \w \b_1 + e_2 \w \b_2 + \left[e_{12} -\frac{1}{H+1} (e_{34}+\ldots+e_{2n-1,2n})\right] \w \b_3 + \b_4~,
\end{align*}
giving us the desired expression.
\end{proof}

With Lemma \ref{pvectd}, we have at hand a recursive algorithm to write down a basis for $P\bw^kV\,$ for $V$ of any arbitrary even dimension.

\subsection{Differential operators and cohomologies}

In this subsection, we shall consider the action of differential operators on $\CL^{r,s}(M)\,$.  We start with the exterior derivative operator, $d\,$.  We obtain  
\begin{align}
d\, \left( \frac{L^r}{r!} B_{s}\right)& = \frac{L^r}{r!}\, d(B_s)\notag\\ 
& = \frac{L^r}{r!}\,B^0_{s+1} + \frac{L^{r+1}}{r!}\,B^1_{s-1} ~. \label{ddecompex}
\end{align}
In the above, we have noted the symplectic condition $[d,L]=0$ and the useful formula \footnote{\eqref{dprim} is just \eqref{dbform} with $\om \w$ replaced here by $L$.}
\be\label{dprim}
dB_s = B^0_{s+1} + L \,B^1_{s-1}~,
\ee
where $B^0_{s+1}\in \CB^{s+1}(M)$ and $B^1_{s-1} \in \CB^{s-1}(M)$ are primitive forms and, moreover, $B^0_{s+1}=0$ if $s=n\,$.  Equation \eqref{dprim} follows simply (see, e.g., \cite{Huy,TY1}) from first writing down the Lefschetz decomposition for $dB_s\,$,  
\be \label{dprimexpl}
dB_s = B^0_{s+1} + L\, B^1_{s-1} + {1\over 2!}\, L^2 \, B^2_{s-3} + {1\over 3!} \, L^3\, B^3_{s-5} + \ldots ~,
\ee
and then applying $L^{n-s+1}$ to both sides of \eqref{dprimexpl}.   The left-hand side will then be zero by the primitive condition, $L^{n-s+1} B_s =0\,$.  Thus, each term on the right-hand side (with now an additional $L^{n-s+1}$) must also be zero.   This results in the requirement that $B^2_{s-3}, B^3_{s-5}, \ldots$ in \eqref{dprimexpl} must be identically zero.    

In all, we have the result that $d$ acting on $\CL^{r,s}$ leads to at most two terms, 
\be\label{cldecomp}
d:~\CL^{r,s}~ \to ~~\CL^{r,s+1} \oplus ~\CL^{r+1,s-1}~.
\ee
As explained in the Introduction (Section 1) via an analogy to the complex geometry case,  \eqref{cldecomp} naturally gives us a decomposition of the exterior derivative operator in symplectic geometry.  And indeed we can define the decomposition of $d$ into two linear differential operators $(\dpp, \dpm)$ by writing
\be\label{ddecomp}
d = \dpp + L \, \dpm~.
\ee
By comparing \eqref{ddecompex} and \eqref{ddecomp}, we have the following definition:

\begin{dfn}\label{pmdef}
On a symplectic manifold $(M, \om)$, the first-order differential operators  $\dpp: \CL^{r,s}(M) \to \CL^{r,s+1}(M)$ and $\dpm: \CL^{r,s}(M) \to \CL^{r,s-1}(M)$ are defined by the property
\begin{align}
\label{dppdef}
\dpp \left(\frac{L^r}{r!}{B_s}\right)& = \, \frac{L^r}{r!}{B^0_{s+1}}~, \\
\label{dpmdef}
\dpm \left(\frac{L^r}{r!}{B_s}\right)& =\, \frac{L^r}{r!}{B^1_{s-1}}~,
\end{align}
where $B_s, B^0_{s+1}, B^1_{s-1} \in \CB^*(M)$ and $dB_s = B^0_{s+1} + L \,B^1_{s-1}\,$.
\end{dfn}

Note that we can restrict to the primitive subspace of differential forms by setting $r=0$ above.  Then $\dpp$ and $\dpm$ become the projections of $(d B_s)$ to the primitive terms, $B^0_{s+1}$ and $B^1_{s-1}\,$, respectively.   Therefore, $\pa_{\pm}: \CB^s(M) \to \CB^{s\pm 1}(M)$ preserve primitivity and are the natural first-order operators on the space of primitive differential forms $\CB^*(M)$.   

With Definition \ref{pmdef}, we have the following properties:
\begin{lem}  \label{pmprop} On a symplectic manifold, $(M^{2n}, \om)$, the symplectic differential operators $(\dpp, \dpm)$ satisfy the following:  (i) $(\dpp)^2 = (\dpm)^2=0$;  (ii) $L(\dpp \dpm) = - L(\dpm \dpp)$; (iii) $[\dpp, L] = [L\,\dpm, \,L] =0$.
\end{lem}
\begin{proof} Using $d=\dpp + L\,\dpm\,$ and the uniqueness of the Lefschetz decomposition, these relations follow directly from $d^2=0$ and $[d,L]=0$.  
\end{proof}

We remark that relations (ii) and (iii) in Lemma \ref{pmprop} simplify to $\dpp \dpm = - \dpm\dpp$ and $[\dpp, L]=[\dpm, L]=0$, respectively, when acting on $\CL^{r,s}(M)$ for $r+s < n$.  Only when $r+s = n$ does the additional $L$ operator need to be present since the primitive condition implies $\dpp\, \CL^{n-s,s} =0$ and also $L\,  \CL^{n-s,s} = 0\,$.  So for the most part, Lemma \ref{pmprop} does imply that $\dpp$ and $\dpm$ besides squaring to zero, also anticommute with each other and commute with $L$.

Let us now consider the symplectic differential operator, $\dl\,$, and its action on $\CL^{r,s}(M)\,$.  Recall that acting on a differential $k$-form, it is defined as
\begin{align*}
\dl :&= d\,\La - \La\, d \\
&= (-1)^{k+1}\ss\,d\,\ss~,
\end{align*}
where in the second line $\dl$ is expressed as the symplectic adjoint of $d$ with respect to the symplectic star operator, $\ss\,$, defined by \eqref{sstardef}.  Using Lemma \ref{llhrelations}(i) and \eqref{dprim}, we can write  
\begin{align*}
d\La \,\frac{L^r}{r!} B_{s}& = (H+R+1) \frac{L^{r-1}}{(r-1)!}B^0_{s+1} + R (H+R)  \frac{L^r}{r!}B^1_{s-1}~,\\
\La d\, \frac{L^r}{r!} B_{s}& = (H+R)\frac{L^{r-1}}{(r-1)!} B^0_{s+1} + (R+1)(H+R) \frac{L^r}{r!} B^1_{s-1}~,
\end{align*}
where, for instance, $(R+1)(H+R) \frac{L^r}{r!} B^1_{s-1}  = (r+1)(n-r -s +1)\frac{L^r}{r!} B^1_{s-1}\,$.  Taking their difference, we obtain
\be\label{dlLr}
\dl \, \frac{L^r}{r!} B_{s} = \frac{L^{r-1}}{(r-1)!} B^0_{s+1} - (H+R) \frac{L^r}{r!} B^1_{s-1} ~,
\ee
which implies
\be\label{lcldecomp}
\dl: ~\CL^{r,s}~ \to ~~\CL^{r-1,s+1} \oplus ~\CL^{r,s-1}~,
\ee 
and the decomposition
\be\label{dldecomp}
\dl = \frac{1}{(H+R+1)}\, \dpp\,\La\, - \,(H+R) \,\dpm  ~,
\ee
where the notation $\dfrac{1}{H+R+1} = (H+R+1)^{-1}$ just inverts the constants, e.g., $(H+R+1)^{-1} \left(\frac{L^r}{r!}B_s\right) = (n-r-s+1)^{-1} \left(\frac{L^r}{r!}B_s\right)\,$.

We can now give an explicit expression for $\dpp$ and $\dpm$ in terms of $d$ and $\dl$.  Comparing \eqref{ddecomp} with \eqref{dldecomp} and using Lemma \eqref{llhrelations}, we obtain the following expressions:
\begin{lem}\label{dppmdeff}On a symplectic manifold $(M, \om)$, $\dpp$ and  $\dpm$ can be expressed as
\begin{align}\label{dppdef}
\dpp &= \frac{1}{H+2R+1}[(H+R+1) d + L\dl]~, \\
\dpm &=\frac{-1}{(H+2R+1)(H+R)}[(H+R) \dl  - \La d\;]~. \label{dpmdef}
\end{align}
\end{lem}

Let us point out first that the operator $(H+2R+1)$ always has a non-zero action on $\CL^{r,s}$, since the corresponding eigenvalue $(n-s+1)>0$ is always positive.   For the operator $1/{(H+R)}$ in \eqref{dpmdef}, it acts on  forms in $\CL^{r,s}$ with $r+s < n$, and thus it is also well-defined.  Now, we could have equivalently defined $\dpp$ and $\dpm$ using the expressions \eqref{dppdef} and \eqref{dpmdef}.   As is straightforward to check, $\dppm$ defined this way satisfy Definition \ref{pmdef}.  Moreover, since $\dl$ is the symplectic adjoint of $d$, i.e.,
$$\dl= (-1)^{k+1} \ss d \, \ss =  (-1)^{k+1} \ss (\dpp + L\, \dpm)\,\ss~, $$
it can also be verified using \eqref{dppdef} and \eqref{dpmdef} that  
\begin{align*}
\dpp^{\ss} &:= (-1)^{k+1} \ss \dpp\, \ss =  \frac{1}{H+R+1} \,\dpp\, \La~, \\
(L\,\dpm)^{\ss}&:=  (-1)^{k+1} \ss (L\,\dpm)\, \ss = - (H+R) \,\dpm ~,
\end{align*}
which are consistent with \eqref{dldecomp}.

With $d$ and $\dl$, we can now proceed to consider their composition, $d\dl$.  This second-order differential operator appears naturally in symplectic cohomologies \cite{TY1}.  We can calculate $d\dl: \CL^{r,s}(M)\to \CL^{r,s}(M)$ using \eqref{ddecomp} and \eqref{dldecomp}.  We find
\begin{align}
d\dl &= \left(\dpp + L\, \dpm \right) \left( \frac{1}{H+R+1}\,\dpp \La - (H+R)\, \dpm\right)\notag \\
& = - \dpp (H+R) \dpm + L\, \dpm \frac{1}{H+R+1}\, \dpp \La \notag\\
& = - \dpp (H+R) \dpm - \frac{1}{H+R+1}\dpp\dpm\, L \La  \notag\\
&= - (H+2R+1) \, \dpp\dpm~,  \label{dppmdef}
\end{align}
where, in the last line, we have used Lemma \ref{llhrelations}(ii).  In short, we have $d\dl \sim \dpp\dpm\,$.

As we have emphasized, the action of the differential operators $(\dpp\,, \dpm\,, \dpp\dpm)$ on $\CL^{r,s}$ reduces to their action on the primitive elements $\CL^{0,s}=\CB^s\,$.  Acting on primitive forms, the expressions for the differential operators simplify.  
\begin{lem} \label{primdiff}Acting on primitive differential forms, the operators, $(\dpp\,, \dpm\,, \dpp\dpm)$ have the expressions 
\begin{align}
\dpp &= d- L\,H^{-1}\La\, d ~,\label{dppP}  \\
\dpm &= \frac{1}{H}\, \La\, d ~,\label{dpmP}\\
\dpp\dpm & =  -\frac{1}{H+1}\, d\,\dl= \frac{1}{H+1}\,d\La d ~,\label{dppmP}
\end{align}
and moreover, $\dl = - H \dpm$.
\end{lem}

And finally, to conclude this subsection, let us note that the elements on the symplectic pyramid can be connected by first-order differential operators as follows:
\bes
\xymatrix{\CL^{r-1,s+1} &  & ~~\CL^{r,s+1} \\
  & \CL^{r,s} \ar[ur]^{\dpp}  \ar[dr]_{L\,\dpm} \ar[dl]^{\dpm} \ar[ul]_{\dpp \,\La}&   \\
  \CL^{r,s-1} &  & ~~~~\CL^{r+1,s-1} ~. 
}
\ees
In the above diagram, the right-pointing arrows with operators $(\dpp\,,\, L\, \dpm)$ are associated with $d\,$, while the left-pointing ones $(\La\,\dpp\,, \,\dpm)$ are associated with $\dl\,$.  From the diagram, we have 
two natural sets of differential complexes.
\bes
\xymatrix@1{ \CL^{r,0} ~ \ar[r]^{\dpp} &~ \CL^{r,1} ~ \ar[r]^{\dpp} &~ ~\ldots~ ~\ar[r]^{\dpp}&~ \CL^{r,n-r-1}~ \ar[r]^{\dpp} &~ \CL^{r,n-r}~,}
\ees
\bes
\xymatrix@1{ \CL^{r,0} ~  &~ \CL^{r,1} ~ \ar[l]_{\dpm} &~ ~\ldots~ ~\ar[l]_{\dpm}&~ \CL^{r,n-r-1}~\ar[l]_{\dpm}&~ \CL^{r,n-r}~\ar[l]_{~\dpm}~.}
\ees
We can construct cohomologies with them.  Define
\bes
H^{r,s}_{\dpp} (M)= \frac{\ker \dpp \cap \CL^{r,s} }{\dpp \CL^{r,s-1}} \ 
\ees
and
\bes
H^{r,s}_{\dpm} (M)= \frac{\ker \dpm \cap \CL^{r,s} }{\dpm \CL^{r,s+1}} \ , 
\ees
for $r<n-s\,$.  But by the commutativity of $\dppm$ with $L$, we have $H^{r,s}_{\dpp}(M) \cong H^{0,s}_{\dpp}(M)$ and $H^{r,s}_{\dpm}(M) \cong H^{0,s}_{\dpm}(M)\,$ for any $r<n-s$.  Hence, we will focus on the two primitive cohomologies
\bes
PH^s_{\dpp} (M)= \frac{\ker \dpp \cap \CB^{s} }{\dpp \CB^{s-1}} \ 
\ees
and
\bes
PH^s_{\dpm} (M)= \frac{\ker \dpm \cap \CB^{s} }{\dpm \CB^{s+1}} \ , 
\ees
for $s<n\,$.

Besides these two cohomologies, let us just note that two other primitive cohomologies were introduced in Paper I \cite{TY1}; they can be found in Table \ref{tabcoh} in Section 1, expressed in terms of $\dpp$ and $\dpm\,$. 

\

\subsection{A symplectic elliptic complex}

We now show that $PH^*_{\dppm}(M)$ is finite dimensional.  Since we have naturally the two differential complexes, 
\bes
\xymatrix@1{ 0~\ar[r] & ~ \CB^{0} ~ \ar[r]^{\dpp} &~ \CB^{1} ~ \ar[r]^{\dpp} &~ ~\ldots~ ~\ar[r]^{\dpp}&~ \CB^{n-1}~ \ar[r]^{\dpp} &~ \CB^{n}~,}
\ees
\bes
\xymatrix@1{0~&~ \CB^{0} ~\ar[l]_{\dpm}  &~ \CB^{1} ~ \ar[l]_{\dpm} &~ ~\ldots~ ~\ar[l]_{\dpm}&~ \CB^{n-1}~\ar[l]_{\dpm}&~ \CB^{n}~\ar[l]_{~\dpm}~,}
\ees
it is interesting to ask whether they are elliptic.  Unfortunately, these two complexes are not elliptic: the ellipticity property breaks down at $\CB^n$ since $\dpp$ maps all primitive $n$-forms to zero, and for $\dpm\,$, there is no primitive $\CB^{n+1}$ space.  We may try to consider connecting the two complexes by joining them as follows:
\bes
\xymatrix@1{\ldots  ~ \ar[r]^\dpp  & ~ \CB^{n-1}~ \ar[r]^{\dpp} &~ \CB^{n} ~\ar[r]^{\dpm} &  ~ \CB^{n-1}~\ar[r]^{\dpm}& ~\ldots.}
\ees
But such a combined complex is unfortunately no longer a differential complex, as $\dpm\dpp\neq 0\,$.  Fortunately, there is a way to obtain a differential elliptic complex if we utilize the second-order differential operator $\dpp\dpm\,$.

\begin{prop}\label{ellipticsc}The following complex is elliptic:\footnote{After proving this proposition, we searched the literature for any mention of such a symplectic elliptic complex.  We found only that the simple $n=2$, $d=4$ case has appeared in \cite{Smith}.  It was presented there as an example of an elliptic complex that does not imply the corresponding local Poincar\'e lemmas (which we also had found and is described here in Proposition \ref{dpplemma}).  Recently, M. Eastwood has informed us that he has also independently arrived at such a complex \cite{Eastwood}.}
\begin{equation}\label{exseq}
\xymatrix@R=30pt@C=30pt{
0\; \ar[r]^\dpp & \; \CB^0 \;\ar[r]^\dpp &\; \CB^1 \;\ar[r]^\dpp& ~ \ldots ~\ar[r]^\dpp&\; \CB^{n-1} \;\ar[r]^\dpp&\; \CB^{n} \;\ar[d]^{\dpp\dpm}\\
0\; & \; \CB^0 \;\ar[l]_\dpm &\; \CB^1 \;\ar[l]_\dpm& ~ \ldots ~\ar[l]_\dpm&\; \CB^{n-1} \;\ar[l]_\dpm&\; \CB^{n} \ar[l]_\dpm \; 
}
\end{equation}
\end{prop}
\begin{proof}
Clearly, the above is a differential complex.  We need to show that the associated symbol complex is exact at each point $x\in M$.   Let $\xi \in T^*_x-\{0\}$.  By an $Sp(2n)$ transformation, we can set $\xi=e_1\,$ and take the symplectic form to be $\om = e_{12} + e_{34} + \ldots + e_{2n-1,2n}$ with $e_1,\dots,e_{2n}$ providing a basis for $T^*_x\,$.  For $\mu_k \in P\!\bw^kT^*_x$, an element in the primitive exterior vector space, we can use the decomposition of Lemma \ref{lpvectd} to write
\be\label{pvectd}
\mu_k = e_1 \w \b_1 + e_2 \w \b_2 + (e_{12} -  \frac{1}{H+1} \sum_{j=2}^n e_{2j-1, 2j}) \w \b_3 + \b_4~,
\ee
where $\b_1,\b_2, \b_3,\b_4\in P\!\bw^*T^*_x$ are primitive exterior products involving only $e_3, e_4,\ldots,e_{2n-1},e_{2n}\,$.  Note that when $k=n$, then $\b_4=0$ since there are no primitive $n$-form without either $e_1$ or $e_2\,$.  

From Lemma \ref{primdiff}, the symbol of the differential operators are given by
\begin{align*}
\s(\dpp)(x,\xi)\, \mu &= (1- L H^{-1} \La) (\xi \w \mu)~, \\
\s(\dpm)(x,\xi)\,\mu &=H^{-1}\La (\xi \w \mu)~,\\
\s(\dpp\dpm)(x,\xi)\,\mu &= (H+1)^{-1} [\xi \w( \La (\xi \w \mu))]~.
\end{align*}
Letting $\xi=e_1$ and $\mu$ taking the form of \eqref{pvectd}, we have that
\begin{align}
\im{\s(\dpp)} &= \left\{(e_{12} -\frac{1}{H+1} \sum_{j=2}^n e_{2j-1, 2j})\w \b_2  ,  e_1 \w \b_4\right\}\label{ima}~,\\
\im{\s(\dpm)} &= \left\{\b_2 ,  e_1 \w \b_3\right\} \label{imb}~,\\
\im{\s(\dpp\dpm)} &= \left\{e_1 \w \b_2\right\} \label{imc}~,
\end{align}
which imply 
\begin{align}
\ker{\s(\dpp)} &= \left\{e_1 \w \b_1 ,  (e_{12} -\frac{1}{H+1} \sum_{j=2}^n e_{2j-1, 2j})\w \b_3 \right\}~, \label{kera}\\
\ker{\s(\dpm)} & = \left\{e_1 \w \b_1 , \b_4\right\}~, \label{kerb} \\
\ker{\s(\dpp\dpm)} & = \left\{e_1\w \b_1 , (e_{12} -\frac{1}{H+1} \sum_{j=2}^n e_{2j-1, 2j})\w \b_3 , \b_4\right\}~.\label{kerc}
\end{align}
Comparing \eqref{kera}-\eqref{kerc} with \eqref{ima}-\eqref{imc}, and noting that for $k=n$, $\b_4=0$, we find that the symbol sequence is exact, i.e., $\ker \s({D_i}) = \im \s(D_{i-1})\,$, as required.
\end{proof}

With an elliptic complex, the associated cohomologies are finite dimensional.  The finiteness of 
$$PH^n_{d+\dl}(M)=\frac{\ker \dpm \cap \CB^n(M)}{\im \dpp\dpm \cap \CB^n(M)}~,\qquad PH^n_{d\dl}(M)=\frac{\ker \dpp\dpm \cap \CB^n(M)}{\im \dpp \cap \CB^n(M)}~, $$
were proved previously in Paper I \cite{TY1}.  But we have now also shown the finiteness of $PH^k_{\dppm}(M)\,$.
\begin{cor} The cohomologies $PH^k_\dpp(M)$ and $PH^k_\dpm(M)$ for $0\leq k<n$ are finite dimensional.
\end{cor}

\section{Properties of $PH_\dppm(M)$}

\subsection{Primitive harmonic forms and isomorphism of $PH_\dpp(M)$ and $PH_\dpm(M)$}

To analyze the properties of $PH^*_\dppm(M)\,$, we shall make use of a compatible triple $(\om, J,  g)$ of symplectic form, almost complex structure, and Riemannian metric, present on all symplectic manifolds.  The Riemannian metric $g$ gives us the standard inner product on differential forms
\be\label{imet}
(A,A') =  \int_M A \w * A' = \int_M g(A, A')\, d{\rm vol} \,, \qquad  A, A' \in \CA^k(M)~.
\ee
With an inner product, we can define the adjoint operators $(\dpps, \dpms)\,$.  They can easily be expressed in terms of $\ds$ and $\dls$ using Lemma \ref{dppmdeff}.
\begin{lem}\label{dppmsdeff}On a symplectic manifold $(M,\om)$ with a compatible Riemannian metric $g$, the adjoints ($\dpps , \dpms$) take the form
\begin{align}
\dpps &= \left[\ds (H+R+1) + \dls \La \right] \frac{1}{H+2R +1}~, \label{dppsdef}\\
\dpms &=- \left[\dls - \ds  \frac{1}{H+R+1}L\right] \frac{1}{H+2R+1}~. \label{dpmsdef}
\end{align}
\end{lem}

\

With the adjoint operators at hand, we can define the associated harmonic forms for $\dppm$ operators.  The natural $\dppm$ Laplacian is the second-order differential operator 
\be\label{dppmLap}
\Delta_\dppm = \dppm (\dppm)^* + (\dppm)^* \dppm~,
\ee
which leads to the following definition:

\begin{dfn}A primitive differential form $B_k \in \CB^k(M)$ for $0\leq k<n$ is called {\it $\dppm$-harmonic} if  $\Delta_{\dppm} B= 0$, or equivalently, 
\be\label{cohahc}
\dppm B_k = 0~, ~~\quad {\rm and} \qquad  (\dppm)^* B_k =0~.
\ee
We denote the space of $\dppm$-harmonic $k$-forms by $P\CH_{\dppm}^k(M)\,$.
\label{dp}
\end{dfn}

Now the elliptic complex \eqref{exseq} implies that $\Delta_{\dppm}$ are elliptic operators.  Thus, applying Hodge theory, we immediately have the following theorem:

\begin{thm}  Let $M$ be a compact symplectic manifold.  For any compatible triple $(\om, J, g)$, we define the standard inner product on $\CB^k(M)$ with respect to $g$.  Then, for $0\leq k<n$:
\begin{itemize}
\setlength{\parsep}{0pt}
\setlength{\itemsep}{0pt}
\item[{\rm(i)}]
$\dim \CH_{\dppm}^k(M) < \infty\,$.
\item[{\rm(ii)}] There is an orthogonal decomposition:
\be\label{cohadecomp}
\CB^k = P\CH^k_{\dppm}  \oplus \dppm \CB^{k\pm 1} \oplus (\dppm)^* \CB^{k\mp 1} \,.
\ee
\item[{\rm(iii)}] There is a canonical isomorphism: $P\CH^k_{\dppm}(M) \cong PH^k_{\dppm}(M)$\,.
\end{itemize}
\label{dppmellip}
\end{thm}

Having demonstrated the finiteness of $PH_\dppm(M)$, let us compare the solution space of $\dppm$-harmonic forms.  We will need to make use of the almost complex structure $J$ and the relation between the Hodge star operator and the symplectic star operator \cite{TY1} given by
\be \label{srels}
* =  \CJ\, \ss ~~,
\ee
where
\bes
\CJ=\sum_{p,q} (\sqrt{-1}\,)^{p-q} \ \Pi^{p,q}
\ees
projects a $k$-form onto its $(p,q)$ parts times the multiplicative factor $(\sqrt{-1}\,)^{p-q}$.  Interestingly, we find that $(\dpp, \dpps)$ is $\CJ$-conjugate to $(\dpms,\dpm)$ up to a non-zero constant.

\begin{lem}\label{pmequiv}
For a compatible triple $(\om, J, g)$ on a symplectic manifold, 
\begin{align}
\CJ \,\dpp \,\CJ^{-1} & = \dpms\,(H+R)~,\label{cj1} \\
\CJ \, \dpps \,\CJ^{-1} & = (H+R)\, \dpm ~.\label{cj2}
\end{align}
\end{lem}
\begin{proof}
Acting on a $k$-form, we have
\begin{align*}
\CJ \,\dpp \,\CJ^{-1} & = \CJ\, \frac{1}{H+2R+1}\left[(H+R+1) d + L \dl \right] \CJ^{-1}\\
& = \frac{(-1)^{k+1}}{H+2R+1}\left[(H+R+1) \CJ \ss \dl \ss \CJ^{-1} + \CJ \ss \La d \ss \CJ^{-1} \right] \\
& = \frac{-1}{H+2R+1}\left[(H+R+1) * \dl *+ * \La d * \right]\\
&=\frac{1}{H+2R+1}\left[-(H+R+1)\dls + L \ds\right]\\
&= (H+R+1) \dpms = \dpms (H+R)~,
\end{align*}
where we have used the expressions for $\dpp$ and $\dpms$ in Lemma \ref{dppmdeff} and Lemma \ref{dppmsdeff}, respectively, and also various relations involving $*$ and $\ss$.  In particular, we applied $d= (-1)^{k+1} (\ss \dl \ss)$ and  $L= \ss \La \ss\,$ in line two,  \eqref{srels} and $\ss \CJ^{-1} = * (-1)^k$ in line three,  and $L = (-1)^k * \La *\,$ in line four.  The equivalence of lines four and five can be checked by explicitly calculating their actions on $\CL^{r,s}\,$.   As for the second equation \eqref{cj2}, it can be derived similarly or interpreted simply as the Hodge adjoint of the first equation \eqref{cj1}.
\end{proof}

Thus, by Lemma \ref{pmequiv}, $B_k \in P^k(M)$ is $\dpp$-harmonic if and only if $\CJ B_k$, which is also primitive, is $\dpm$-harmonic.  This implies that the two harmonic spaces are isomorphic, and moreover, by Theorem \ref{dppmellip}(iii), that the two respective primitive cohomologies are also isomorphic.

\begin{prop}\label{dppmiso}
Let $(M,\om)$ be a compact symplectic manifold and let $0\leq k<n\,$.  Then $P\CH^k_{\dpp}(M)\cong P\CH^k_{\dpm}(M)$ and $PH^k_{\dpp}(M)\cong PH^k_{\dpm}(M)\,$.
\end{prop}

Coupled with the isomorphism of $PH^n_{d+\dl}(M) \cong PH^n_{d\dl}(M)$ \cite{TY1}, we find that the analytical index of the elliptic complex \eqref{exseq} is trivial.  

\begin{cor}
The index of the elliptic complex of \eqref{exseq} is zero.
\end{cor}

Let us note further that the isomorphism between $PH^k_\dpp(M)$ and $PH^k_{\dpm}(M)$ leads to a natural pairing between the two cohomologies, similar to that for $PH^n_{d+\dl}(M)$ and $PH^n_{d\dl}(M)$ described in Paper I \cite[Prop.\, 3.26]{TY1}.

\begin{prop}On a compact symplectic manifold $(M,\om)$, there is a natural pairing 
\bes
PH^k_{\dpp}(M) \otimes PH^{k}_{\dpm}(M) \longrightarrow  \mathbb{R}~
\ees
defined by 
\bes
B_k \otimes B'_k \longrightarrow \int_M \frac{1}{(n-k)!}\,\om^{n-k} \w B_k \w B'_k~,
\ees
which is non-degenerate.
\end{prop}
\begin{proof} Let us first interpret the integral.  Combining \eqref{ssrel} and \eqref{srels}, we obtain the well-known relation (see, e.g., \cite{Huy})
\bes\label{Weil}
*\,\frac{1}{r!}L^r B_k = (-1)^{\frac{k(k+1)}{2}} \frac{1}{(n-k-r)!} \, L^{n-k-r} \CJ(B_k)~.
\ees
Hence, the integral can be rewritten as 
\bes
\int_M \frac{1}{(n-k)!}\,\om^{n-k} \w B_k \w B'_k = (-1)^{\frac{k(k+1)}{2}} \int_M B_k \w * (\CJ^{-1}B'_k)~.
\ees
In this form and noting Lemma \ref{pmequiv}, it is clear that the pairing is well defined since the integral is independent of the choice of the representatives of the two cohomology classes.   Now to show non-degeneracy, we can choose $B_k$ and $B'_k$ to be the respective harmonic representatives.  In particular, let $B_k\in P\CH^k_{\dpp}(M)$ and $B'_k=\CJ B_k \in P\CH^{k}_{\dpm}(M)$.  We thus have for $B_k\neq 0$
\bes
B_k\otimes  B'_k \longrightarrow \int_M \frac{1}{(n-k)!}\,\om^{n-k} \w B_k \w B'_k  = (-1)^{\frac{k(k+1)}{2}}\| B_k\|^2  \neq 0~.
\ees
\end{proof}

\

\subsection{Local primitive Poincar\'e lemmas}

We now consider local Poincar\'e lemmas for the various cohomologies we have studied.  Except for cohomologies of degree zero forms and the cohomology $PH^1_\dpp$ and $PH^1_{d\dl}\,$, all other local primitive cohomologies turn out to be trivial.  At the end of this subsection, we shall use the $\dpm$-Poincar\'e lemma to demonstrate the equivalence of $PH_\dpm(M)$ with the \v Cech cohomology of $\CB'^n(M)$, where $\CB'^k(M)$ denotes the space of $\dpm$-closed primitive $k$-forms.  

On a open unit disk, the Poincar\'e lemma states that only $H^0_d(U)$ is non-empty.  By the symplectic star operation, there is also the $\dl$-Poincar\'e lemma 

\begin{lem}{\rm ($\dl$-Poincar\'e lemma).} \label{pdllemma}Let $U$ be an open unit disk in $\mathbb{R}^{2n}$ and $\om= \sum dx^{i} \w dx^{i+n}$, the canonical symplectic form.  If $A_k\in \CA^k(U)$ is $\dl$-closed and $k< 2n$, then there exists a $A'_{k+1}\in \CA^{k+1}(U)$ such that $A_k = \dl A'_{k+1}$.  
\end{lem}
\begin{proof}  Let ${\tA}_{2n-k}= \ss\, A_k$.  Then $\dl A_k = (-1)^{k+1} \ss d \ss \, A_k  = (-1)^{k+1} \ss d {\tA}_{2n-k}=0$.  By the Poincar\'e lemma, we can write $\tA _{2n-k} = (-1)^k d \tA'_{2n-k-1}$, where the additional $(-1)^k$ factor is inserted for convenience.  Then, letting $A'_{k+1} = \ss \tA'_{2n-k-1}$, we have
\bes
A_k = \ss\, \tA_{2n-k} = (-1)^k \ss d {\tA'}_{2n-k-1}  = (-1)^k \ss d \ss A'_{k+1} = \dl A'_{k+1}~.
\ees
\end{proof}

\begin{prop}{\rm (Primitive Poincar\'e lemma).} \label{pdlemma}
Let $U$ be an open unit disk in $\mathbb{R}^{2n}$ and $\om= \sum dx^i \w dx^{i+n}$, the canonical symplectic form.  If $B_k\in P^k(U)$ is $d$-closed and $0<k\leq n$, then there exists a form $B'_{k-1}\in P'^{k-1}(U)$ such that $B_k = d B'_{k-1}$.
\end{prop}
\begin{proof}  
By the Poincar\'e lemma, there exists  a $(k-1)$-form with the property $B_k=d A_{k-1}$.  We give a standard construction of $A_{k-1}$ (see, e.g.,  \cite[Appendix 5]{Petersen}) and show that $A_{k-1}$ turns out to be primitive and $\dpm$-closed.

Start with the radial vector field $V= x^i \pa_i$.  Such a vector fields only scales differential forms.  For instance, $\CL_V \om = 2 \om$.   Hence, a primitive differential form remains primitive under a diffeomorphism generated by $V$.  Acting on a primitive $d$-closed form, we have 
\bes
\CL_V B_k = d\, i_V B_k ~.
\ees
Note that $i_V B_k$ is also primitive.  Moreover, since $\CL_V B_k$ remains primitive, this implies that $i_V B_k \in P'^{k-1}(U)$.   

We introduce the operator $T: \CA^k \to \CA^k$, which is inverse to the Lie derivative $\CL_V$ and commutes with $d$, 
\bes
T\; \CL_V = id~,\qquad \qquad d\, T = T \, d ~.
\ees 
It can be checked \cite[p. 385]{Petersen} that such a $T$ is given by
\bes
T \left(\frac{1}{k!}\, A_{i_1 \ldots i_k} \ dx^{i_1} \w \ldots \w dx^{i_k}\right) = \frac{1}{k!} \left(\int_0^1 t^{k-1}  A_{i_1 \ldots i_k}(tx)\, dt \right) dx^{i_1} \w \ldots \w dx^{i_k}~.
\ees
With these properties, we can write
\bes
B_k = (T\, \CL_V) B_k =  T\, d (i_V B_k) = d( T\, i_V B_k)~.
\ees
As mentioned, $i_V B_k$ is a primitive $(k-1)$-form, $L^{n-k+2} (i_V B_k) =0$.  Clearly, we also have $L^{n-k+2} \,T (i_V B_k)=0$ and so $B'_{k-1}:= T (i_V B_k)$ must also be primitive.  Last, since $dB'_{k-1}\in P^k(U)$, this implies $B'_{k-1}\in P'^{k-1}(U)$.
\end{proof}

\begin{prop}{\rm (Primitive $d\dl$-Poincar\'e lemma).}  \label{pddllemma}Let $U$ be an open unit disk in $\mathbb{R}^{2n}$ and $\om= \sum dx^i \w dx^{i+n}$, the canonical symplectic form.  If $B_k\in \CB^k(U)$ is $d$-closed and $0<k\leq n$, then there exists a $B''_{k}\in \CB^{k}(U)$ such that $B_k = d\dl B''_{k}$.  
\end{prop}
\begin{proof}
By Proposition \ref{pdlemma}, since $B_k$ is $d$-closed, we can write $B_k= dB'_{k-1}$ for some $B'_{k-1}\in P'^{k-1}(U)$.  But since $\dl B'_{k-1} =0$, by the $\dl$-Poincar\'e lemma, there exists $A''_{k}\in \CA^k(U)$ such that $B'_{k-1}= \dl A''_k$ and hence $B_k = d\dl A''_k$.  But this implies by Lemma 3.9 of \cite{TY1} that there exists a primitive $k$-form $B''_k$ such that $B_k = d\dl B''_k$.    
\end{proof}

Lefschetz decomposition and the commutativity of $d\dl$ with the $sl(2)$ representation $(L, \La, H)$ \cite{TY1} then implies that $d\dl$-Poincar\'e lemma holds for all differential forms.

\begin{cor}{\rm (Local $d\dl$ lemma).} Let $U$ be an open unit disk in $\mathbb{R}^{2n}$ and $\om= \sum dx^i \w dx^{i+n}$, the canonical symplectic form.  If $A_k\in \CA^k(U)$ is $d+\dl$-closed and $k>0$, then there exists a $A'_{k}\in \CA^{k}(U)$ such that $A_k = d\dl A'_{k}$.  
\end{cor}

\begin{prop}{\rm (Primitive $(\dpp+\dpm)$-Poincar\'e lemma).} \label{pdpdllemma} Let $U$ be an open unit disk in $\mathbb{R}^{2n}$ and $\om= \sum dx^i \w dx^{i+n}$, the canonical symplectic form.  Then,  $\dim PH^k_{d\dl}(U)= 0$ for $k=0$ and $2\leq k \leq n\,$, while  $\dim PH^1_{d\dl}(U) = 1\,$.
\end{prop}
\begin{proof}
For $k=0\,$, the $\dl$-Poincar\'e lemma (Lemma \ref{pdllemma}) implies that any $B_0\in P^0(U)$ can be expressed as $B_0 = \dpm B_1\,$ for some $B_1\in P^1(U)$.

Now for $2\leq k \leq n\,$, let $B_k\in P^k(U)$ be $d\dl$-closed.  Let $B_{k-1}= \dl B_k$.  Since $dB_{k-1}=0$, the $d\dl$-Poincar\'e lemma (Proposition \ref{pddllemma}) implies that
\bes
B_{k-1} = d\dl B'_{k-1} = - \dl (d B'_{k-1})~.
\ees
Notice that $B_k + d B'_{k-1}$ is then $\dl$-closed.  The $\dl$-Poincar\'e lemma then implies $B_k + d B'_{k-1} = \dl A''_{k+1}$, or equivalently, $B_k = -dB'_{k-1} + \dl A''_{k+1}$.  But then by Lemma 3.20 of \cite{TY1}, we can write $B_k = \dpp \hB'_{k-1} + \dpm \hB''_{k+1}$.

For $k=1$, let $B_1\in P^1(U)$ be $d\dl$-closed.  If $\dl B_1 =0$, then we can write $B_1=\dl A''_2 = \dpp \hB'_{0} + \dpm \hB''_{2}$, arguing similarly as in the $2\leq k \leq n$ case.  Now, if $\dl B_1 = B_0 \neq 0$, then such a $B_1$ cannot be $(\dpp+\dpm)$ exact, for any exact one-form $B_1 = \dpp \hB'_{0} + \dpm \hB''_{2}$ has the property
\bes
\dl B_1 = -n\, \dpm\dpp \hB'_{0} = n\, \dpp\dpm \hB'_{0} = 0~,
\ees
using \eqref{dldecomp}.  But with $\dl B_1 = B_0 \neq 0$ and $d\dl B_1 = d B_0 =0$, $B_0$ must be some constant c.  Furthermore, if 
both $B_1^A$ and $B_1^B$ are $d\dl$-closed and 
$\dl B_1^A = \dl B^B_1 =c\,$, then by the $\dl$-Poincar\'e lemma, their difference $B_1^A - B_1^B$ must be exact, i.e., $B_1^A - B_1^B= \dl A'' = \dpp \hB'_{0} + \dpm \hB''_{2}\,$.  Hence, we can conclude that $\dim PH^1_{d\dl}(U) =1\,$. 
\end{proof}

\begin{prop}{\rm (Primitive $\dpp$-Poincar\'e lemma).} \label{dpplemma}Let $U$ be an open unit disk in $\mathbb{R}^{2n}$ and $\om= \sum dx^i \w dx^{i+n}$, the canonical symplectic form.  Then $\dim PH^0_{\dpp}(U) =\dim PH^1_{\dpp}(U)=1$ and $\dim PH^k_{\dpp}(U)=0$ for $2\leq k < n\,$.
\end{prop}

\begin{proof}
The $k=0$ case is just the standard $d$-Poincar\'e lemma.

For $0<k<n\,$, let $B_k\in\CB^k(U)$ be $\dpp$-closed.  Then either (1) $dB_k=0$ or (2) $\dpp B_k =0$ but $d B_k = L B^1_{k-1}\neq 0$.  In the case of (1), it follows from the primitive Poincar\'e lemma (Proposition \ref{pdlemma}) that there exists a $B_{k-1}\in \CB^{k-1}(U)$ such that $B_k= \dpp B'_{k-1}$.  So we only need to consider case (2), which we will analyze in two parts. 

(2a) Let $2\leq k < n$.  Since $d B_k = L B^1_{k-1}$, we have 
$$d^2B_k = L \,dB^1_{k-1}=L B^{10}_k =0~.$$
Since $LB^{10}_k$ cannot be identically zero unless $k=n$, we find that $dB^1_{k-1}=0$.  Now by the primitive Poincar\'e lemma, $B^1_{k-1}= d \hB_{k-2}$.  Thus, $dB_k= LB^1_{k-1}$ implies
$$ d(B_k - L \hB_{k-2}) = 0  \qquad \Longrightarrow \qquad B_k - L \hB_{k-2} = d\tA_{k-1}~.$$
Lefschetz decomposing $\tA_{k-1} = \tB_{k-1} + L \tB_{k-3} + \ldots\,$, it is clear that $B_k = \dpp \tB_{k-1}$.

(2b) Let $k=1$.  If $dB_1 = L B^1_0$, then clearly $B_1 \neq dB_0 = \dpp B_0$.  But with $d^2 B_1 = L d B^1_0 =0$, which implies $dB^1_0 =0$, i.e., $B^1_0$ is a constant.  This gives us a one-parameter space for $PH^1_{\dpp}(U)$.  For if both $dB^A_1 = dB^B_1 = LB_0^1$, it follows from the $d$-Poincar\'e lemma that $B^A_1=B^B_1 + dB^{AB}_0 = B^B_1 + \dpp B^{AB}_0$.  Thus, $B^A_1$ and $B^B_1$ are in the same class in $PH^1_{\dpp}(U)$. 
\end{proof}

Let us note that the the non-trivial representative of $PH^1_{d\dl}(U)$ and $PH^1_{\dpp}(U)$ is just the tautological one-form.  We now turn to $PH_{\dpm}(U)$ which interestingly differs from $PH_{\dpp}(U)$.

\begin{prop} {\rm (Primitive $\dpm$-Poincar\'e lemma).} \label{dpmlemma}Let $U$ be an open unit disk in $\mathbb{R}^{2n}$ and $\om= \sum dx^i \w dx^{i+n}$, the canonical symplectic form.  Then $\dim PH^k_{\dpm}(U) =0$ for $0\leq k< n\,$.
\end{prop}

\begin{proof}
For $k=0$, this is just the $\dl$-Poincar\'e lemma.  For $0<k<n\,$, let $B_k\in\CB^k(U)$ be $\dpm$-closed.  Then either (1) $dB_k=0$ or (2) $\dpm B_k =0$, but $d B_k = B^0_{k+1}\neq 0\,$.  In case (1), it follows from the primitive $d\dl$-Poincar\'e lemma (Proposition \ref{pddllemma}) that there exists a $B'_{k+1}\in \CB^{k+1}(U)$ such that $B_k= \dpm B'_{k+1}$.  

For case (2), with $d B_k = B^0_{k+1}$, clearly $B^0_{k+1}$ is $d$-closed.  Hence, by the primitive $d\dl$-lemma, we can write $d B_k= B^0_{k+1}= d\dl \tB_{k+1}$.   This means that $d(B_k - \dl \tB_{k+1}) = 0\,$.  Applying the primitive $d\dl$-lemma again to $B_k - \dl \tB_{k+1}\,$, we have $B_k - \dl \tB_{k+1} = \dpp\dpm B'_{k}$ for some $B'_k\in \CB^k(U)\,$.  Letting $B'_{k+1}=- \dpp B'_k\,$ and using \eqref{dldecomp}, we find
$$B_k = -(n-k) \dpm \tB_{k+1} + \dpm B'_{k+1} = \dpm (-(n-k)\tB_{k+1} + B'_{k+1})~.$$  
\end{proof}

With $\CB'(M)$ denoting the space of primitive forms that are $\dpm$-closed, the $\dpm\,$-\:Poincar\'e lemma implies the exactness of the following sequence of primitive sheaves $\CB\,$:
\begin{equation}\label{dpmexact}
\xymatrix@1{
0\; \ar[r]&\; \CB'^n\; \ar[r]^{{\it i}} \;& \; \CB^n \;\ar[r]^{\dpm} &\;  \CB^{n-1}\; \ar[r]^{\dpm} &~ \ldots ~\ar[r]^{\dpm} & \;\CB^1 \;\ar[r]^{\dpm} & \;\CB^0\;\ar[r]^{\dpm} & \;0~.
}
\end{equation}
Since the $\CB^k$ allows for partition of unity, they are fine sheaves, and thus, the \v Cech cohomology ${\breve H}^l(M,\CB^k)=0$ for $l>0$.  Then, by standard arguments, we have the following:
\begin{thm} For $(M^{2n},\om)$ a compact symplectic manifold,
$$ PH^{k}_\dpm(M)\cong {\breve H}^{n-k}(M,\CB'^n)\qquad\qquad {\rm for~}0\leq k < n ~.$$
\end{thm}

\

\subsection{Comparing $PH_\dppm(M)$ with $H_{d}(M)$ and $H_{\dl}(M)$  and the $\dpp\dpm$-lemma}

Let us note that all zero-forms and one-forms are primitive forms.  
Therefore, we may expect that the $PH^k_\dppm(M)$ may be equivalent to one of the standard cohomology at low degree.  Indeed, this is the case, as the following proposition shows.
\begin{prop}\label{deglow}
On a compact symplectic manifold $(M,\om)$, we have the following equivalence
\begin{align*}
PH^k_\dpp(M)= H^k_d (M)~,\quad PH^k_\dpm(M)= H^k_{\dl} (M)~, \qquad {\rm for~} k=0,1~, 
\end{align*}
where $H^k_d(M)$ is the de Rham cohomology and $H^k_{\dl}(M)= (\ker \dl \cap \CA^k(M)) / (\im \dl \cap \CA^k(M))\,$.
\end{prop}
\begin{proof}
Note first that the action of $\dpm$ on zero- and one-forms is identical to that of $\dl$ modulo a non-zero constant (i.e., $-1/H$).  For $\dpp$, the action on zero-forms is identical to $d$.  So the equivalence at degree $k=0$ is trivial.  

For $PH^1_\dpp(M)$, note first that one forms that are $\dpp$-exact are also $d$-exact.  So the question is whether there are any $\dpp$-closed one-forms that are not $d$-closed.  Now, if $B_1\in \CB^1(M)$ is $\dpp$-closed, then we can have $dB_1= L B^1_0\,.$  Furthermore, $d^2=0$ implies $dB^1_0=0$, which means that $B^1_0=c\,$, a constant.  However, $c=0$ since otherwise the symplectic form would be trivial in de Rham cohomology.  Thus, we find that $PH^1_\dpp(M)= H^1_d (M)$, having used the compactness of $M$.

For $PH^1_\dpm(M)$, as mentioned, $\dpm$-closed one-forms are also $\dl$-closed.  Moreover, $\dpm$-exact one-forms are also trivially $\dl$-exact.  We shall now show that any $\dl$-exact one-forms are in fact also $\dpm$-exact.   Let  $B_1$ be $\dl$-exact; i.e.,
\begin{align*}
B_1&= \dl A_2= \dl (B_2 + L B_0)\\
&= -(n-1) \dpm B_2 + d B_0~.
\end{align*}
We therefore need to show that there exists a $B'_2\in \CB^2(M)$ such that $d B_0=\dpm B'_2\,$.  To do this, we can assume $\int_M B_0=0\,$, (since, if necessary, we can always subtract a constant factor from $B_0$ without affecting $d B_0$).  With $B_0$ integrating to zero and trivially $\dl$-close, $B_0$ must be $\dl$-exact, i.e., $B_0 = \dl B'_1\,$.  Clearly then, we now have
\begin{align*}
B_1&=-(n-1) \dpm B_2 + d( \dl B'_1) \\
&=\dpm\left[ -(n-1) B_2 + n\, \dpp B'_1\right]  ~,
\end{align*}
where we have used the relation $d\dl=  (H+2R+1)\dpm\dpp = \dpm (H+2R+2) \dpp\,$.   Therefore, we have shown that a $\dl$-exact one-form is also $\dpm$-exact, and this completes the proof.
\end{proof}

So at degree $k=0,1$, we have that $PH_\dpp^k(M)$ is equivalent to the de Rham cohomolgy and $PH^k_\dpm(M)$ to the $\dl$-cohomology.   At higher degree, the equivalence generally does not hold any longer.  To maintain some kind of equivalence, we can assume additional conditions on $M$.  A useful condition is the $d\dl$-lemma.  Recall that we say that the $d\dl$-lemma holds on a symplectic manifold $(M,\om)$ if it satisfies the following condition: Let $A\in \CA^*(M)$ be a differential form that is both $d$- and $\dl$-closed, then either it is not exact or else it must be $d-$, $\dl$-, and $d\dl$-exact.  Now since we are dealing with only primitive forms, it is not difficult to show that the $d\dl$-lemma for $\CA(M)$ is equivalent to the $\dpp\dpm$-lemma defined below for $\CB(M)$ on a compact symplectic manifold.

\begin{dfn}{ [$\dpp\dpm$-lemma]} On a symplectic manifold $(M,\om)$, let $B_k\in\CB^k(M)$ be $d$-closed.  We then say that the $\dpp\dpm$-lemma holds if the following properties are equivalent:
\begin{itemize}
\setlength{\itemsep}{0pt}
\item[{\rm(i)}] $B_k$ is $\dpp$-exact.
\item[{\rm(ii)}] $B_k$ is $\dpm$-exact if $k<n\,$.
\item[{\rm(iii)}] $B_k$ is $\dpp\dpm$-exact if $k>0\,$.
\end{itemize}
\label{dppmlem}
\end{dfn}

Interestingly, it has been shown by Merkulov \cite{Merkulov} and Guillemin \cite{Guillemin} (see also Cavalcanti \cite{Caval}) that a symplectic manifold exhibits the $d\dl$-lemma (or equivalently, the $\dpp\dpm$-lemma) if and only if the the strong Lefschetz property holds.  Here, strong Lefschetz is the property that the map in de Rham cohomology  $\varphi: H_{d}^k(M) \to H_{d}^{2n-k}(M)$ given by $A_k \to [\om]^{n-k} \w A_k$ is an isomorphism for all $k\leq n\,$.   Imposing the $\dpp\dpm$-lemma or the strong Lefschetz, we have the following property for $PH^k_{\dppm}(M)$:

\begin{prop}On a symplectic manifold $(M,\om)$, if the $\dpp\dpm$-lemma holds, or equivalently the strong Lefschetz property holds, then for $2 \leq k < n\,$, we have
\begin{align}
PH^k_\dpp(M) &= H^k_{d}(M) \cap \CB^k(M): = \dfrac{\ker d \cap \CB^k(M)}{d \CA^{k-1} \cap \CB^k(M)}~,\label{dpphd}\\
PH^k_\dpm(M) &= H^k_{\dl}(M) \cap \CB^k(M): = \dfrac{\ker \dl \cap \CB^k(M)}{\dl \CA^{k+1} \cap \CB^k(M)}~.\label{dpmhdl}
\end{align}
\end{prop}
\begin{proof}Consider first $PH^k_\dpp(M)\,$.  If $B_k\in \CB^k(M)$ is $\dpp$-closed, then in general we have $dB_k = L B^1_{k-1}\,$.  Now since $dB^1_{k-1}=0$ and $B^1_{k-1}=\dpm B_k$, we can use the $\dpp\dpm$-lemma to write $B^1_{k-1}= \dpp\dpm B'_{k-1}\,$ for some $B'_{k-1}\in \CB^{k-1}(M)$.  Therefore, we have
$$d(B_k + \dpp B'_{k-1}) = L(B^1_{k-1} + \dpm \dpp B'_{k-1})= 0~,$$  
implying that in every cohomology class of $PH^k_\dpp(M)$, there must exist a representative that is also $d$-closed, having assumed of course that the $\dpp\dpm$-lemma holds. 

We now only need to show that for a primitive $d$-closed form $B_k$, if $B_k =\dpp B_{k-1}\,$, then there exists an $A_{k-1}\in \CA^{k-1}(M)$ such that $B_k = d A_{k-1}\,$.  But since $B_k$ is $d$-closed and $\dpp$-exact, it must also be $\dpp\dpm$-exact, or equivalently, $d\dl$-exact.  Therefore, we must have $B_k =d(\dl B'_k)$ for some $B'_k\in \CB^k(M)\,$, which  completes the proof of \eqref{dpphd}.

Consider now $PH^k_{\dpm}(M)\,$.  Acting on primitive forms, $\dl: \CB^k(M) \to \CB^{k-1}(M)$.  Therefore, $ \dpm$-closedness is in fact equivalent to $\dl$-closedness.  Moreover, a $\dpm$-exact form is also $\dl$-exact, but the converse is generally not true.  Thus we have to show that a $\dl$-exact form can also be expressed as a $\dpm$-exact form if the $\dpp\dpm$-lemma holds. 

Let $B_k\in \CB^k(M)$ be $\dl$-exact (i.e., $B_k = \dl A_{k+1}$).  Then, since $B_k$ is $\dl$-closed (and equivalently, $\dpm$-closed), we must either have (i) $d B_k=0\,$ or (ii) $d B_k = B^0_{k+1}\,$.  For case (i), $B_k$ satisfies the criteria for the $d\dl$-lemma, and so we can immediately write, $B_k = \dpm(\dpp B'_k)\,$, for some $B'_k\in \CB^k(M)\,$, noting again that $d\dl \sim \dpp\dpm\,$.  For case (ii), clearly $dB^0_{k+1}=0\,$; therefore, we can apply the $\dpp\dpm$-lemma to $B^0_{k+1}$ and write $B^0_{k+1} = \dpp\dpm B'_{k+1}\,$, for some $B'_{k+1}\in \CB^{k+1}(M)\,$.   We thus obtain
\bes
d(B_k - \dpm B'_{k+1})=0 ~,
\ees
which reduces the problem to case (i).  Thus applying the $\dpp\dpm$-lemma again, we find that $B_k = \dpm B'_{k+1} + \dpm (\dpp B'_k)\,$, for some $B'_k\in \CB^k(M)\,$.
\end{proof}

As mentioned in the above proof, $\ker \dpm = \ker \dl \cap \CB\,$, but in general $\im \dpm \subset \im \dl \cap \CB\,$.  This thus give a lower bound on the dimension of the $PH^k_\dppm(M)\,$.

\begin{prop}
On a compact symplectic manifold $(M,\om)$, 
$$\dim PH^k_\dpp(M) = \dim PH^k_\dpm(M) \geq \dim \left(H^k_{\dl}(M) \cap \CB^k(M)\right)~.$$
\end{prop}

\

\

\section{Example: A symplectic nilmanifold}

We can explicitly calculate and compare the different primitive cohomologies on a six-dimensional compact symplectic nilmanifold.  Let $M=M^6$  be the nilmanifold of type $(0,0,0,12,14,15+23+24)\,$.  This means that there exists a basis of one-forms $\ea,\eb,\ldots,\ef$ on $M$ with the following alegbra:
\begin{align}
d\ea & = 0~,  \qquad\qquad\qquad  d\ed = \ea \w \eb~, \notag\\
d\eb & = 0~,  \qquad\qquad\qquad d\eee  = \ea \w \ed~,  \label{cotvect}\\
d\ec & = 0~,  \qquad\qquad\qquad d\ef  =  \ea\w\eee + \eb\w \ec + \eb\w \ed ~.\notag
\end{align}
This nilmanifold has the Betti numbers $(b_1, b_2, b_3) = (3,5,6)$ \cite{Salamon}.  For our calculation, let us take the symplectic form to be
\be\label{exsymp}
\om = \ea \w \ef + \eb \w \eee - \ec \w \ed ~.
\ee
It can be easily checked that $\om$ of \eqref{exsymp} is both $d$-closed and non-degenerate, as required.  
In Table \ref{extable}, we give the basis elements for the cohomologies $H_{d}(M)\,$, $H_{\dl}(M)$, $PH_\dppm(M)$, and $PH_{d+\dl}(M)\,$.
\begin{table}
{\renewcommand{\arraystretch}{1.25} 
\begin{tabular}{c |c | l | l | l }
 & $k=0$ &~~$k=1$ & ~~~~~~~~~ $k=2$ & ~~~~~~~~~$k=3$ \\
\hline
$H^k_{d}$ &  $1$ & $\ea,\eb,\ec$&$ \om,e_{13}, (e_{23}-e_{24}),  $ &  $\om \w \eb, \om \w \ec, (e_{315}+e_{415}), e_{425},$ \\
& &  &$(e_{15}-e_{23}),(e_{26}-e_{45}) $& $(e_{534}+ e_{623}), (e_{516}+e_{534}-2\,e_{263}+e_{624})$ \\
\hline
$H^k_{\dl}$& $1$ &$\ed,\eee,\ef$& $\om, e_{46}, (e_{15}-e_{23}),$ &$\om \w \eb, \om \w \ec, (e_{315}+e_{415}), e_{416},$ \\ 
& & &$(e_{26}-e_{45}),(e_{35}+e_{45})$  & $(e_{516}+ e_{623}), (e_{516}+e_{534}-2\,e_{263}+e_{624})$ \\
\hline
$PH^k_{\dpp}$& $1$ & $\ea,\eb,\ec$& $e_{13},\! (e_{23}\!-\!e_{24}),\! (e_{15}\!-\!e_{23}),$ &\\
& & &$(e_{26}-e_{45}),(e_{35}-e_{45})$  &  \\
\hline
$PH^k_{\dpm}$& $1$ & $\ed,\eee,\ef$& $e_{24}, e_{46}, (e_{15}-e_{23}),$ &\\
& & &$(e_{26}-e_{45}),(e_{35}+e_{45})$  &  \\
\hline
$PH^k_{d+\dl}$ & $1$ & $ \ea,\eb,\ec$  &$  e_{12},e_{13},e_{14}, e_{24},  $& $e_{315}, e_{415}, (e_{125}+e_{134}), (e_{126}-e_{234}), $\\
& &  &$(e_{15}-e_{23}), (e_{26}-e_{45}),$ &   $(e_{316}-e_{325}+2\,e_{416}-2\,e_{425}),$\\
& &  &$(e_{15}+e_{23}+e_{24})$ & $(e_{516}+e_{534}-2\,e_{263}+e_{624}) $\\
\end{tabular}}

\ 
 \caption{Bases for $H_{d}$, $H_{\dl}$, $PH_{\dpp}$, $PH_{\dpm}$, and $PH_{d+\dl}$
of the six dimensional nilmanifold in terms of exterior products of the one-forms $e_i$ \eqref{cotvect} and symplectic form $\om$ \eqref{exsymp}.\label{extable}}
\
 \end{table}

Clearly, the $\dpp\dpm$-lemma generally does not hold for this nilmanifold.   Take for instance, $e_{12}$.  It is primitive, $d$-closed, and $\dpp$-exact, i.e. $e_{12} = \dpp e_4\,$.  Moreover, it is also $\dpm$-exact, since $e_{12}=\dpm(e_{416}-e_{425})\,$.  However, it is not $\dpp\dpm$-exact. 

Notice for $k=2$, $\dim PH^2_{\dpp}(M) = \dim [H^2_{d} \cap \CB^2(M)]+1$.  The difference is due to the presence of the two-form $(e_{35}-e_{45})$, which is $\dpp$-closed but not $d$-closed.  Explicitly, we have 
$$ d(e_{35}-e_{45}) = e_{134} - e_{125} = - \om \w e_1 ~.$$
Hence, we see that the map $\varphi: H^1(M)\to H^3(M)$ given by $[\om] \w$ is not injective for this nilmanifold with $\om$ of \eqref{exsymp}.  
Similarly, $\dim PH^2_{\dpm}(M) = \dim[H^2_{\dl} \cap \CB^2(M)] +1\,$.  This is due to the fact that $e_{24}\in \im \dl$ but not in the image of $\dpm$.  Specifically, we have
$$ \dl [(e_{625}+e_{634}) + \om \w \ef] = 2\, e_{24}~,$$  
where the presence of the non-primitive term $\dl (\om \w \ef) = \dpp \ef = e_{15}+ e_{23}+ e_{24}$ is essential.  Effectively, we have a primitive $d$-closed two-form $B_2=e_{15}+ e_{23}+ e_{24} = \dpp \ef \neq \dpm B_3\,.$ 

Now, we could have chosen a different symplectic form.  For instance, consider the same nilmanifold but with the symplectic form given by
\be
\om' = \ea \w \ec + \eb \w \ef - \ed \w \eee~.
\ee

In this case, it is easy to show that the map $\varphi: H^1(M) \to H^3(M)$ now using $[\om'] \w$ is injective.  Furthermore, any closed primitive two-form, if $\dpp$-exact, is also $\dpm$-exact.  In this case, we have $PH^2_\dpp(M,{\om'}) = H^2_{d}\cap \CB^2(M,\om')$ and $PH^2_\dpm(M,{\om'}) = H^2_{\dl}\cap \CB^2(M,\om')\,$.  And moreover, we have 
$$\dim PH^2_{\dppm}(M,\om) = \dim PH^2_{\dppm}(M,\om') +1~,$$ 
which shows that $PH^2_{\dppm}(M)$ can vary as the de Rham class of the symplectic form is varied.

\

\vskip 1cm
Department of Mathematics, University of California, Irvine, CA 92697 USA\\
{\it Email address:}~{\tt lstseng@math.uci.edu}
\vskip .5 cm
\noindent 
{Department of Mathematics, Harvard University\\
Cambridge, MA 02138 USA}\\
{\it Email address:}~{\tt yau@math.harvard.edu}

\end{document}